\documentclass[11pt]{amsart}

\usepackage[T1]{fontenc}
\usepackage[margin=1in]{geometry}
\usepackage{amsthm}
\usepackage{amssymb}

\theoremstyle{plain}
\newtheorem{theorem}[subsection]{Theorem}
\newtheorem{lemma}[subsection]{Lemma}
\newtheorem{proposition}[subsection]{Proposition}
\newtheorem{corollary}[subsection]{Corollary}

\theoremstyle{definition}
\newtheorem{definition}[subsection]{Definition}

\newtheorem{remark}[subsection]{Remark}

\title{Stability of Partitions Induced by Nearest-Center Assignment Under Perturbations}

\author{MD Nahidul Hasan Sabit}
\address{Princeton University, USA}
\email{nahid.sabit@princeton.edu}

\author{Faija Anjum}
\address{United World College, Mostar, Bosnia \& Herzegovina}
\email{faija.anjum@uwcmostar.ba }

\date{}

\begin{document}
\maketitle

\begin{abstract}
We study clustering through the partitions it induces on a finite labeled set $[n]=\{1,\dots,n\}$, and analyze how these partitions change under perturbations of a point configuration $X=(x_1,\dots,x_n)\in(\mathbb{R}^d)^n$. We equip the space of partitions $\Pi_n$ with a normalized pairwise disagreement metric $d(\cdot,\cdot)$, and define the stability radius $r(X,A)=\sup\{\varepsilon\ge0: A(X')=A(X)\ \text{whenever}\ \|X-X'\|\le\varepsilon\}$, where $\|X-X'\|=\max_i\|x_i-x'_i\|$. Our main results concern nearest-center assignment with fixed centers $\{c_1,\dots,c_k\}\subset\mathbb{R}^d$. For each point, we define the margin $\gamma_i=\min_{j\ne\ell_i}(\|x_i-c_j\|-\|x_i-c_{\ell_i}\|)$ and $\gamma_{\min}=\min_i\gamma_i$, where $\ell_i$ denotes the assigned center. We show that if $\varepsilon<\gamma_{\min}/2$, then no assignments change under perturbation and hence $A(X')=A(X)$. Conversely, any point that changes its assigned center must satisfy $\gamma_i\le2\varepsilon$, showing that instability is localized near decision boundaries. We construct configurations in which arbitrarily small perturbations $\|X-X'\|\le\varepsilon$ alter the induced partition, demonstrating that the margin condition is sufficient but not necessary for stability. We further extend the framework to a discrete-time setting, showing that if $\sum_t \delta_t<r(X(0),A)$, then $A(X(t))=A(X(0))$, and we give a probabilistic bound on $\mathbb{E}[d(A(X),A(X'))]$ in terms of tail probabilities relative to $\gamma_i$. This framework identifies the margin as the key quantity governing both worst-case and average stability, and provides explicit conditions under which clustering-induced partitions remain invariant in a fixed-center Euclidean model.
\end{abstract}

\section{Introduction}

Clustering methods assign each point in a finite dataset to a group. The output of such a method can be viewed as a partition of the index set $[n] = {1, \dots, n}$ of the data points. In many applications, these groupings are used to support decisions. For example, hospitals group patients by risk, and financial institutions group customers by behavior. These assignments influence treatment, pricing, and allocation of resources. In practice, the input data often contains noise or measurement error. Small changes in the data may be modeled as perturbations of a point configuration $X = (x_1, \dots, x_n) \in (\mathbb{R}^d)^n$ to a nearby configuration $X'$. Even when the perturbation size $|X - X'|$ is small, the resulting assignments may change, and hence may alter the induced partition of $[n]$. This can affect the decisions that depend on these groupings. This motivates the following question: how stable is the partition produced by a clustering rule under perturbations of the data? In this paper, we study this question by focusing on the partitions induced by clustering rules. We treat the output of a clustering rule as a partition of a finite labeled set and study how this partition changes when the input configuration is perturbed. To measure changes, we use a normalized pairwise disagreement between partitions and define a notion of stability radius, which captures the largest perturbation size under which the induced partition remains unchanged. We note that a change in assigned centers may or may not change the induced partition, since partitions record only which indices are grouped together and do not encode the labels of the centers. In particular, simultaneous reassignment of entire groups may preserve the partition structure. Throughout the paper, we therefore distinguish between changes in assignments and changes in the induced partition.

\section{Related Background and Prior Work}

This work is focused on the clustering methods, stability under perturbation, and the combinatorics of partitions.

Clustering methods take a finite set of points and assign them to groups. The output of such a method can be viewed as a partition of the index set $[n]$ of the data points. Many standard approaches, including $k$-means and hierarchical clustering, can be interpreted in this way. These methods are widely used in applications where grouping influences downstream decisions \cite{1,2}.

There is also a line of work on clustering stability. These works study how the output of a clustering method changes when the input data is modified. The goal is often to evaluate robustness or to guide model selection. Much of this literature is algorithmic or statistical in nature and depends on properties of the specific method being used \cite{3,4}.

From a combinatorial perspective, a clustering can be viewed as a partition of a finite set. Partitions are fundamental objects in discrete mathematics and appear in areas such as graph theory and set systems \cite{5}. Measures of similarity or dissimilarity between partitions, including pairwise agreement and disagreement, are commonly used in clustering evaluation \cite{6}.

The approach in this paper differs from these directions in two ways. First, we fix a clustering rule and study the induced partition directly, rather than analyzing an optimization procedure or a randomized algorithm. This removes sources of variability unrelated to the geometry of the data. Second, we express stability entirely in terms of partitions and a metric on $\Pi_n$, allowing the problem to be studied through pairwise relations on the index set.

Within this framework, we identify the margin as the key quantity governing stability. We give explicit deterministic conditions under which the induced partition remains unchanged under perturbations, characterize which points can change their assigned centers, and construct configurations where arbitrarily small perturbations alter the partition. We also extend the analysis to discrete-time evolution and give bounds on expected partition change under random perturbations.

The goal is not to model all aspects of clustering stability, but to isolate a setting in which the role of margin can be described precisely and without additional algorithmic or statistical assumptions.

\section{Preliminaries and Definitions}

We fix notation and basic definitions used throughout the paper. All sets are finite. We assume $n \ge 2$ so that $\binom{n}{2} > 0$.
We work with a finite labeled point configuration
\[
X = (x_1, x_2, \ldots, x_n), \qquad x_i \in \mathbb{R}^d.
\]
We emphasize that $X$ is an ordered $n$-tuple rather than a set. The labeling by indices is part of the data and is essential throughout, since partitions are defined on the index set
$[n] = \{1,2,\ldots,n\},$
and perturbations compare each $x_i$ with a corresponding perturbed point $x_i'$ having the same index.

\begin{definition}[Partition]
A partition of $[n]$ is a collection
$P = \{C_1, C_2, \ldots, C_k\}$
of nonempty subsets of $[n]$ such that
\[
C_i \cap C_j = \emptyset \text{ for } i \ne j, 
\qquad \bigcup_{i=1}^k C_i = [n].
\]
Each set $C_i$ is called a block of the partition. We write $\Pi_n$ for the set of all partitions of $[n]$.
\end{definition}

\begin{definition}[Clustering rule]
A clustering rule is a map
$A : (\mathbb{R}^d)^n \to \Pi_n$
that assigns to each finite labeled point configuration $X$ a partition $A(X)$ of $[n]$. Two indices $i,j \in [n]$ lie in the same block of $A(X)$ if and only if the corresponding points $x_i$ and $x_j$ are assigned to the same cluster under the rule.
\end{definition}

We stress that the clustering rule acts on labeled configurations. Thus, two configurations that differ by a permutation of indices are treated as distinct inputs.

\begin{definition}[Partition distance]
Let $P,Q \in \Pi_n$. For $1 \le i < j \le n$, define
\[
\mathbf{1}_P(i,j) =
\begin{cases}
1 & \text{if } i \text{ and } j \text{ lie in the same block of } P, \\
0 & \text{otherwise.}
\end{cases}
\]
The partition distance is defined by
$d(P,Q) = \frac{1}{\binom{n}{2}} 
\left| \left\{ (i,j) : 1 \le i < j \le n,\ \mathbf{1}_P(i,j) \ne \mathbf{1}_Q(i,j) \right\} \right|.$

Thus $d(P,Q)$ is the fraction of unordered pairs $\{i,j\}$ on which the two partitions disagree about whether the indices belong to the same block. The quantity $d(P,Q)$ takes values in $[0,1]$, and $d(P,Q)=0$ if and only if $P=Q$.
\end{definition}

\begin{definition}[Perturbation]
Let
\[
X = (x_1,\ldots,x_n), \qquad X' = (x_1',\ldots,x_n')
\]
be two finite labeled point configurations indexed by the same set $[n]$. The perturbation size is defined by
\[
\|X - X'\| = \max_{1 \le i \le n} \|x_i - x_i'\|.
\]
This measures the largest displacement of any point, with indices matched across the two configurations.
\end{definition}

\begin{definition}[Stability radius]
Let $A$ be a clustering rule and let $X$ be a finite labeled point configuration. The stability radius of $X$ with respect to $A$ is
\[
r(X,A) = \sup \left\{ \varepsilon \ge 0 : A(X') = A(X)\ \text{for all } X' \text{ with } \|X - X'\| \le \varepsilon \right\}.
\]
Equivalently, $r(X,A)$ is the largest perturbation size for which the induced partition on $[n]$ remains unchanged.
\end{definition}

The definitions above describe the general framework used in the paper. We begin with a labeled configuration, apply a clustering rule to obtain a partition, and measure how this partition changes under perturbations. In later sections, we specialize to nearest-center assignment with fixed centers in Euclidean space, where the structure of the rule allows explicit stability bounds to be derived.

\section{Deterministic Stability: Two Clusters}

In this section, we study a specific clustering rule: nearest-center assignment with two fixed centers in $\mathbb{R}^d$. The centers are fixed and are not affected by perturbations of the data.

Let
$X = (x_1, x_2, \ldots, x_n) \in (\mathbb{R}^d)^n,$
and let $c_1, c_2 \in \mathbb{R}^d$ be two distinct centers. We define a clustering rule $A$ by assigning each point to its nearest center.
\[
x_i \in C_1 \ \text{if } \|x_i - c_1\| \le \|x_i - c_2\|, \qquad
x_i \in C_2 \ \text{otherwise}.
\]
When a point is equidistant from $c_1$ and $c_2$, we assign it to $C_1$. This gives a deterministic tie-breaking rule, ensuring that the induced partition $A(X)$ is well-defined.

The decision boundary is the set of points equidistant from the two centers. We introduce a quantity that measures how stable each point’s assignment is.

\begin{definition}[Margin]
For each $x_i \in X$, define
$\gamma_i = \left| \|x_i - c_1\| - \|x_i - c_2\| \right|.$
Define the minimum margin
$\gamma_{\min} = \min_{1 \le i \le n} \gamma_i.$ The quantity $\gamma_i$ measures the difference between the distances of $x_i$ to the two centers. Points with small $\gamma_i$ lie close to the decision boundary and are sensitive to perturbations, while points with large $\gamma_i$ are more stable.
\end{definition}

\begin{theorem}
Let
\[
X' = (x_1', x_2', \ldots, x_n')
\]
be a perturbation of $X$ such that
\[
\|X - X'\| \le \varepsilon.
\]
If
$\varepsilon < \frac{\gamma_{\min}}{2},$
no point changes its assigned center under the perturbation. Consequently,
$A(X') = A(X).$
\end{theorem}

\begin{proof}
Fix $i \in \{1, \ldots, n\}$ and let $x_i'$ be the perturbed point. By definition of the perturbation size,
$\|x_i - x_i'\| \le \varepsilon.$
We compare the distances from $x_i'$ to the centers. By the triangle inequality, for $j = 1,2$,
\[
\|x_i' - c_j\| \le \|x_i - c_j\| + \varepsilon, \qquad
\|x_i' - c_j\| \ge \|x_i - c_j\| - \varepsilon.
\]

Assume that $x_i$ is assigned to $c_1$, so that
\[
\|x_i - c_1\| \le \|x_i - c_2\|.
\]
Then
\[
\|x_i - c_2\| - \|x_i - c_1\| = \gamma_i \ge \gamma_{\min}.
\]

For the perturbed point, we have
\[
\|x_i' - c_2\| - \|x_i' - c_1\|
\ge (\|x_i - c_2\| - \varepsilon) - (\|x_i - c_1\| + \varepsilon)
= \gamma_i - 2\varepsilon.
\]

Since $\varepsilon < \gamma_{\min}/2 \le \gamma_i/2$, we obtain
\[
\gamma_i - 2\varepsilon > 0.
\]
Thus,
\[
\|x_i' - c_2\| > \|x_i' - c_1\|,
\]
so $x_i'$ remains assigned to $c_1$.

The same argument applies if $x_i$ is assigned to $c_2$. Therefore, no point changes its assigned center under the perturbation. Since cluster membership is determined by assigned centers and no assignments change, all pairwise relations of belonging to the same cluster remain unchanged. Thus, the induced partition does not change, and $A(X') = A(X)$.
\end{proof}

This result shows that the minimum margin provides a sufficient condition for stability in the two-center setting. If all points are sufficiently far from the decision boundary, small perturbations do not change any assignments, and the induced partition remains unchanged.

\section{Structural Bounds and Generalizations}

In this section, we extend the results from the two-center case to the setting of $k \ge 2$ fixed centers. We continue to work with nearest-center assignment in $\mathbb{R}^d$, where the centers are fixed under perturbations and a deterministic tie-breaking rule is used.

\smallskip

Let
$X = (x_1, x_2, \ldots, x_n) \in (\mathbb{R}^d)^n.$
Let $c_1, c_2, \ldots, c_k \in \mathbb{R}^d$ be distinct centers. Each point is assigned to its nearest center, producing a partition $A(X)$ of $[n]$. For each $i$, let $\ell_i \in \{1, \ldots, k\}$ denote the index of the center to which $x_i$ is assigned.

We introduce a margin that measures how close a point is to changing its assigned center.

\begin{definition}[Margin]
For each $x_i \in X$, define
$\gamma_i = \min_{j \ne \ell_i} \left( \|x_i - c_j\| - \|x_i - c_{\ell_i}\| \right),$
and define
$\gamma_{\min} = \min_{1 \le i \le n} \gamma_i.$
The quantity $\gamma_i$ measures how close the point $x_i$ is to switching its assigned center. Points with small margin lie near decision boundaries, while points with large margin are more stable.
\end{definition}

\begin{lemma}[Boundary-margin criterion]
Let
\[
X' = (x_1', x_2', \ldots, x_n')
\]
satisfy
\[
\|X - X'\| \le \varepsilon.
\]
If
$\varepsilon < \frac{\gamma_i}{2},$
$x_i$ does not change its assigned center under the perturbation.
\end{lemma}

\begin{proof}
Let $x_i'$ be the perturbed point. For any $j \ne \ell_i$, by the triangle inequality,
\[
\|x_i' - c_j\| \ge \|x_i - c_j\| - \varepsilon, \qquad
\|x_i' - c_{\ell_i}\| \le \|x_i - c_{\ell_i}\| + \varepsilon.
\]
Thus
\[
\|x_i' - c_j\| - \|x_i' - c_{\ell_i}\|
\ge (\|x_i - c_j\| - \|x_i - c_{\ell_i}\|) - 2\varepsilon
\ge \gamma_i - 2\varepsilon.
\]
If $\varepsilon < \gamma_i/2$, then $\gamma_i - 2\varepsilon > 0$, so for all $j \ne \ell_i$,
\[
\|x_i' - c_j\| > \|x_i' - c_{\ell_i}\|.
\]
Thus $x_i'$ remains assigned to $c_{\ell_i}$.
\end{proof}

We now obtain a global stability condition.

\begin{proposition}[Global no-switch condition]
Let $X'$ satisfy $\|X - X'\| \le \varepsilon$. If
$\varepsilon < \frac{\gamma_{\min}}{2},$
no point changes its assigned center, and consequently
$A(X') = A(X).$
\end{proposition}

\begin{proof}
For each $i$, we have $\gamma_i \ge \gamma_{\min}$, so
\[
\varepsilon < \frac{\gamma_{\min}}{2} \le \frac{\gamma_i}{2}.
\]
By the boundary-margin criterion, no point changes its assigned center. Since cluster membership is determined by assigned centers and no assignments change, all pairwise relations remain unchanged. Therefore, $A(X') = A(X)$.
\end{proof}

We now describe which points may change their assigned centers under perturbation.

\begin{proposition}
Let $X'$ satisfy $\|X - X'\| \le \varepsilon$. If $x_i$ changes its assigned center under the perturbation, then
$\gamma_i \le 2\varepsilon.$
\end{proposition}

\begin{proof}
We argue by contrapositive. If $\gamma_i > 2\varepsilon$, then $\varepsilon < \gamma_i/2$, and by the boundary-margin criterion, $x_i$ does not change its assigned center. Therefore, any point that changes its assigned center must satisfy $\gamma_i \le 2\varepsilon$.
\end{proof}

This shows that assignment changes are localized near decision boundaries: only points with small margin can change their assigned center.

We now relate the margin to the stability radius defined earlier.

\begin{corollary}
The stability radius satisfies
\[
r(X, A) \ge \frac{\gamma_{\min}}{2}.
\]
\end{corollary}

\begin{proof}
If $\varepsilon < \gamma_{\min}/2$, then by the global no-switch condition,
\[
A(X') = A(X)
\]
for all $X'$ with $\|X - X'\| \le \varepsilon$. By the definition of the stability radius, this implies
\[
r(X, A) \ge \frac{\gamma_{\min}}{2}.
\]
\end{proof}

These results show that the margin controls stability in nearest-center assignment with fixed centers. Points that are far from decision boundaries are stable under small perturbations, while points with small margin are the primary source of assignment changes. The bounds identify when stability is guaranteed and which points may change their assigned centers, but they do not characterize all possible cases of instability.

\section{Instability Constructions and Counterexamples}

The previous sections give sufficient conditions under which the induced partition is stable. We now show that these conditions cannot be removed in general. In particular, we construct configurations where arbitrarily small perturbations cause points to change their assigned centers, and in these constructions this leads to changes in the induced partition.

We continue to work with nearest-center assignment with fixed centers in $\mathbb{R}^d$, together with a deterministic tie-breaking rule.

We consider a simple setting in $\mathbb{R}^2$. Let
\[
c_1 = (-1,0), \qquad c_2 = (1,0).
\]
The decision boundary is the vertical line $x = 0$.

\subsection{Single-point instability with anchored clusters}

\begin{proposition}
For any $\varepsilon > 0$, there exists a configuration $X$ with $n \ge 3$ and a perturbation $X'$ with $\|X - X'\| \le \varepsilon$ such that at least one index changes its assigned center, and the induced partition differs.
\end{proposition}

\begin{proof}
Let $\varepsilon > 0$ and choose $\delta > 0$ such that $2\delta < \varepsilon$.

Define three points
\[
x_1 = (-2,0), \qquad x_2 = (2,0), \qquad x_3 = (\delta,0),
\]
and define the configuration $X = (x_1,x_2,x_3) \in (\mathbb{R}^2)^3$. Then $x_1$ is assigned to $c_1$, while $x_2$ and $x_3$ are assigned to $c_2$.

Label the indices as $1,2,3$ corresponding to $x_1,x_2,x_3$. The induced partition is
$P = \{\{1\}, \{2,3\}\}.$
Now, we define a perturbation by moving only $x_3$.
$x_3' = (-\delta,0),$
and keep $x_1, x_2$ unchanged. Then
\[
\|X - X'\| = \max_i \|x_i - x_i'\| = 2\delta < \varepsilon.
\]

After perturbation, $x_3'$ is assigned to $c_1$, while $x_1$ remains assigned to $c_1$ and $x_2$ to $c_2$. Thus the induced partition becomes
$P' = \{\{1,3\}, \{2\}\}.$

At least one index changes its assigned center, and the block structure differs, so $P \ne P'$. This shows that arbitrarily small perturbations can change the induced partition.
\end{proof}

\subsection{Many-point instability with anchored clusters}

\begin{proposition}
For any $\varepsilon > 0$ and any integer $m \ge 1$, there exists a configuration $X$ and a perturbation $X'$ with $\|X - X'\| \le \varepsilon$ such that at least $m$ indices change their assigned centers, and the induced partition differs.
\end{proposition}

\begin{proof}
Let $\varepsilon > 0$ and choose $\delta > 0$ such that $2\delta < \varepsilon$.

Let
\[
c_1 = (-1,0), \qquad c_2 = (1,0).
\]
Define anchor points
\[
a_1 = (-2,0), \qquad a_2 = (2,0),
\]
and for $i = 1,\ldots,m$, define
$x_i = (\delta,i).$ Each $x_i$ is assigned to $c_2$, while $a_1$ is assigned to $c_1$ and $a_2$ to $c_2$. Label the indices as $1,2,\ldots,m+2$, where $1$ corresponds to $a_1$, $2$ to $a_2$, and $3,\ldots,m+2$ to $x_1,\ldots,x_m$. The induced partition is
$P = \{\{1\}, \{2,3,\ldots,m+2\}\}.$
Now, we define
\[
x_i' = (-\delta,i), \qquad i = 1,\ldots,m,
\]
and keep $a_1,a_2$ fixed. Then,
$\|X - X'\| = 2\delta < \varepsilon.$
After perturbation, all $x_i'$ are assigned to $c_1$. Thus the induced partition becomes
$P' = \{\{1,3,\ldots,m+2\}, \{2\}\}.$ At least $m$ indices change their assigned centers. The partition differs, so $P \ne P'$.
\end{proof}

\subsection{Instability from near-boundary configurations}

\begin{proposition}
There exist configurations with arbitrarily small margin such that arbitrarily small perturbations change the induced partition.
\end{proposition}

\begin{proof}
Let
\[
c_1 = (-1,0), \qquad c_2 = (1,0),
\]
and fix anchor points
\[
a_1 = (-2,0), \qquad a_2 = (2,0).
\]
Let $\delta > 0$ be small and define
$x = (\delta,0).$
Then $x$ is assigned to $c_2$. Label the indices as $1,2,3$ corresponding to $a_1,a_2,x$. The induced partition is
$P = \{\{1\}, \{2,3\}\}.$
Now, we define
$x' = (-\delta,0).$
Then,
$\|x - x'\| = 2\delta,$
which can be made arbitrarily small, and $x'$ is assigned to $c_1$. The induced partition becomes
$P' = \{\{1,3\}, \{2\}\}.$
Thus, $P \ne P'$. The partition changes under arbitrarily small perturbations.
\end{proof}

These constructions show that the margin condition provides a sufficient condition for stability, but cannot be removed in general. Instability arises when some points lie close to the decision boundary, while other points remain well separated and anchor the cluster structure. In such settings, small perturbations can move points across the boundary, causing assignment changes that may alter the induced partition.

We emphasize that these examples show that assignment changes can lead to changes in the induced partition. However, assignment changes do not in general imply a change in the partition, since partitions record only which indices are grouped together and do not encode the labels of the centers.

\section{Discrete-Time Extension}

We extend the framework to a discrete-time setting in which the data evolves step by step. The goal is to study how the induced partition changes over time under successive perturbations. We consider a sequence of finite labeled point configurations
$X(0), X(1), X(2), \ldots,$
where for each $t \ge 0$,
$X(t) = \big(x^{(t)}_1, x^{(t)}_2, \ldots, x^{(t)}_n\big) \in (\mathbb{R}^d)^n.$

We assume that the indexing of points is fixed over time. For each $i$, the sequence $x^{(t)}_i$ represents the same labeled point as it evolves. This allows us to compare configurations using the same index set.
Let $A$ be a clustering rule. For each $t$, define the induced partition
$P^{(t)} = A\big(X(t)\big).$
This produces a sequence of partitions. We study how this sequence changes over time.

\begin{definition}[One-step perturbation size]
For each $t \ge 0$, define
\[
\delta_t = \|X(t+1) - X(t)\| = \max_{1 \le i \le n} \|x^{(t+1)}_i - x^{(t)}_i\|.
\]
\end{definition}

\begin{definition}[Instability time]
Let $\eta > 0$. Define
\[
\tau_\eta = \inf \left\{ t \ge 1 : d\big(P^{(0)}, P^{(t)}\big) \ge \eta \right\}.
\]
If no such time exists, we set $\tau_\eta = \infty$.
\end{definition}

The value $\tau_\eta$ is the first time at which the partition differs from the initial partition by at least $\eta$ in partition distance.

We first bound the total change between two time steps.

\begin{lemma}[Cumulative perturbation bound]
For any integers $0 \le s < t$,
\[
\|X(t) - X(s)\| \le \sum_{r=s}^{t-1} \delta_r.
\]
\end{lemma}

\begin{proof}
Fix $i \in \{1,\ldots,n\}$. By repeated application of the triangle inequality,
\[
\|x^{(t)}_i - x^{(s)}_i\|
\le \sum_{r=s}^{t-1} \|x^{(r+1)}_i - x^{(r)}_i\|.
\]
Taking the maximum over $i$, we obtain
\[
\|X(t) - X(s)\| = \max_i \|x^{(t)}_i - x^{(s)}_i\|
\le \max_i \sum_{r=s}^{t-1} \|x^{(r+1)}_i - x^{(r)}_i\|.
\]
Since the maximum of a sum is bounded by the sum of maxima,
\[
\|X(t) - X(s)\|
\le \sum_{r=s}^{t-1} \max_i \|x^{(r+1)}_i - x^{(r)}_i\|
= \sum_{r=s}^{t-1} \delta_r.
\]
\end{proof}

This shows that the total drift from time $s$ to time $t$ is controlled by the sum of the one-step perturbations. We now relate cumulative perturbation to stability.

\begin{proposition}[Persistence under cumulative control]
Let $r(X(0),A)$ be the stability radius of the initial configuration. If
\[
\sum_{r=0}^{t-1} \delta_r < r(X(0),A),
\]
then
\[
P^{(t)} = P^{(0)}.
\]
\end{proposition}

\begin{proof}
By the cumulative perturbation bound,
\[
\|X(t) - X(0)\| \le \sum_{r=0}^{t-1} \delta_r.
\]
If this quantity is strictly smaller than $r(X(0),A)$, then by the definition of the stability radius,
\[
A\big(X(t)\big) = A\big(X(0)\big).
\]
Thus $P^{(t)} = P^{(0)}$.
\end{proof}

\begin{corollary}[Lower bound on instability time]
Let $\eta > 0$. If
\[
\sum_{r=0}^{t-1} \delta_r < r(X(0),A),
\]
then
\[
\tau_\eta > t.
\]
\end{corollary}

\begin{proof}
By the proposition, $P^{(t)} = P^{(0)}$, so
\[
d\big(P^{(0)}, P^{(t)}\big) = 0 < \eta.
\]
Thus the first time at which the partition distance reaches $\eta$ must occur after time $t$.
\end{proof}

We also record a local version of stability.

\begin{proposition}[Stepwise stability]
Suppose that for each $r = 0,1,\ldots,t-1$,
\[
\delta_r < r\big(X(r),A\big).
\]
Then
\[
P^{(0)} = P^{(1)} = \cdots = P^{(t)}.
\]
\end{proposition}

\begin{proof}
Fix $r$. Since
\[
\|X(r+1) - X(r)\| = \delta_r < r\big(X(r),A\big),
\]
the definition of the stability radius implies
\[
A\big(X(r+1)\big) = A\big(X(r)\big).
\]
Thus $P^{(r+1)} = P^{(r)}$. Iterating this argument for $r=0,\ldots,t-1$ gives
\[
P^{(0)} = P^{(1)} = \cdots = P^{(t)}.
\]
\end{proof}

These results give two complementary perspectives on stability over time. The first is global: if the cumulative perturbation from the initial configuration remains below the initial stability radius, then the initial partition is preserved. The second is local. If each step is small relative to the current stability radius, then the partition does not change at that step.

The fixed indexing assumption allows perturbations to be measured consistently over time. Under this assumption, the discrete-time framework provides a direct way to track when the partition changes and how cumulative perturbations control long-term behavior.

\section{A Simple Probabilistic Extension}

We give a simple probabilistic consequence of the margin-based argument. The goal is to describe how likely it is that points change their assigned centers under random perturbations. We continue to work with nearest-center assignment with fixed centers in $\mathbb{R}^d$, together with a deterministic tie-breaking rule.

\smallskip

Let
$X = (x_1, \ldots, x_n) \in (\mathbb{R}^d)^n.$ Let $c_1, \ldots, c_k \in \mathbb{R}^d$ be fixed centers. For each $i$, let $\ell_i$ be the index of the center to which $x_i$ is assigned, and define the margin
\[
\gamma_i = \min_{j \ne \ell_i} \left( \|x_i - c_j\| - \|x_i - c_{\ell_i}\| \right).
\]

We consider random perturbations of the form
$x_i' = x_i + \eta_i,$
where $\eta_1, \ldots, \eta_n$ are independent random vectors in $\mathbb{R}^d$. Let
\[
X' = (x_1', \ldots, x_n'), \qquad P = A(X), \qquad P' = A(X').
\]

\subsection{Pointwise switching bound}

\begin{proposition}
If the point $x_i$ changes its assigned center under the perturbation, then
\[
\|\eta_i\| \ge \frac{\gamma_i}{2}.
\]
Consequently,
\[
\mathbb{P}(\text{$x_i$ changes its assigned center}) 
\le \mathbb{P}\!\left( \|\eta_i\| \ge \frac{\gamma_i}{2} \right).
\]
\end{proposition}

\begin{proof}
If $\|\eta_i\| < \gamma_i/2$, then $\|x_i - x_i'\| < \gamma_i/2$. By the boundary-margin criterion from Section 5, the point $x_i$ does not change its assigned center. Therefore, if $x_i$ changes its assigned center, we must have $\|\eta_i\| \ge \gamma_i/2$. Taking probabilities gives the result.
\end{proof}

This reduces switching of a point to a tail event for the perturbation magnitude.

\subsection{Expected number of switched points}

Let $N_{\mathrm{sw}}$ be the number of indices whose assigned centers change under the perturbation. Then
$N_{\mathrm{sw}} = \sum_{i=1}^n I_i,$
where $I_i$ is the indicator of the event that $x_i$ changes its assigned center.

\begin{proposition}
We have
\[
\mathbb{E}[N_{\mathrm{sw}}] \le \sum_{i=1}^n 
\mathbb{P}\!\left( \|\eta_i\| \ge \frac{\gamma_i}{2} \right).
\]
\end{proposition}

\begin{proof}
By linearity of expectation,
\[
\mathbb{E}[N_{\mathrm{sw}}] = \sum_{i=1}^n \mathbb{E}[I_i] 
= \sum_{i=1}^n \mathbb{P}(\text{$x_i$ changes its assigned center}).
\]
Applying the previous proposition to each term gives the result.
\end{proof}

\subsection{Connection to partition distance}

We now relate the number of switched indices to the partition distance.

\begin{proposition}
Let $P = A(X)$ and $P' = A(X')$ arise from nearest-center assignment with the same fixed centers and the same deterministic tie-breaking rule. If at most $N_{\mathrm{sw}}$ indices change their assigned centers, then
\[
d(P, P') \le \min\!\left(1, \frac{2N_{\mathrm{sw}}}{n-1}\right).
\]
\end{proposition}

\begin{proof}
Each index that changes its assigned center can affect at most $(n-1)$ unordered pairs involving that index. Thus the total number of pairs whose same-block relation may change is at most $N_{\mathrm{sw}}(n-1)$. Dividing by $\binom{n}{2}$ gives
\[
d(P, P') \le \frac{N_{\mathrm{sw}}(n-1)}{\binom{n}{2}} 
= \frac{2N_{\mathrm{sw}}}{n-1}.
\]
Since $d(P,P') \le 1$ by definition, the stated bound follows.
\end{proof}

\begin{corollary}
We have
\[
\mathbb{E}[d(P,P')] \le \frac{2}{n-1} 
\sum_{i=1}^n \mathbb{P}\!\left( \|\eta_i\| \ge \frac{\gamma_i}{2} \right).
\]
\end{corollary}

This shows that the expected change in the partition is controlled by the margins through tail probabilities of the perturbations.

\subsection{Special cases}

We record two cases where the bounds become explicit.

\begin{corollary}[Bounded noise]
Assume that $\|\eta_i\| \le \rho$ almost surely for every $i$. If
\[
\gamma_i > 2\rho,
\]
then
\[
\mathbb{P}(\text{$x_i$ changes its assigned center}) = 0.
\]
In particular, if $\gamma_{\min} > 2\rho$, then
\[
\mathbb{P}(P' = P) = 1.
\]
\end{corollary}

\begin{proof}
If $\|\eta_i\| \le \rho < \gamma_i/2$, then by the boundary-margin criterion, the point cannot change its assigned center.
\end{proof}

\begin{corollary}[Gaussian noise]
Assume that $\eta_i \sim \mathcal{N}(0, \sigma^2 I_d)$ independently for each $i$. Then for each $i$,
\[
\mathbb{P}(\text{$x_i$ changes its assigned center}) 
\le \mathbb{P}\!\left( \|\eta_i\| \ge \frac{\gamma_i}{2} \right).
\]
\end{corollary}

\begin{remark}
Under Gaussian noise, $\|\eta_i\|$ has a chi-type distribution. Standard tail bounds imply that
\[
\mathbb{P}\!\left( \|\eta_i\| \ge \frac{\gamma_i}{2} \right)
\]
decays rapidly as $\gamma_i^2 / \sigma^2$ increases (for fixed dimension $d$). Thus points with large margin are unlikely to change their assigned centers, while points with small margin are more likely to switch.
\end{remark}

This section shows that the margin provides not only a sufficient condition for worst-case stability, but also a way to control average behavior under random perturbations. Points with large margin are unlikely to change their assigned centers, while points with small margin contribute most to expected changes. In general, assignment changes may or may not alter the induced partition, but the expected partition distance can be bounded in terms of these switching probabilities.

\section{Illustrative Examples}

This section provides simple examples and simulation setups that illustrate the definitions and bounds developed in the previous sections. These examples are intended to give intuition for how the induced partition behaves under perturbations, rather than to provide formal statistical results.

We work in $\mathbb{R}^2$ and use nearest-center assignment with fixed centers and a deterministic tie-breaking rule.

\subsection{Data model and perturbations}

We consider a finite labeled point configuration
\[
X = (x_1, \ldots, x_n) \in (\mathbb{R}^2)^n,
\]
with two fixed centers
\[
c_1 = (-1,0), \qquad c_2 = (1,0).
\]
Points are generated near the centers. For each index $i \in [n]$, we choose a center $c_{\ell_i}$ and sample
$x_i = c_{\ell_i} + \xi_i,$
where $\xi_i \sim \mathcal{N}(0,\sigma_0^2 I_2)$. Typical parameters are $n = 200$ and $\sigma_0 = 0.2$.
We consider two perturbation models. In the bounded perturbation model, for a fixed $\varepsilon > 0$, we sample $\eta_i$ uniformly from the disk of radius $\varepsilon$ and define
$x_i' = x_i + \eta_i.$
In the Gaussian perturbation model, for a fixed $\sigma > 0$, we take
$\eta_i \sim \mathcal{N}(0,\sigma^2 I_2)$
independently and define $x_i' = x_i + \eta_i$. In both cases, we form
\[
X' = (x_1', \ldots, x_n'), \qquad P = A(X), \qquad P' = A(X').
\]

\subsection{Stable configuration}

Consider a configuration in which all points are well separated from the decision boundary $x=0$, so that the margins $\gamma_i$ are uniformly large. If $\varepsilon < \gamma_{\min}/2$, then by the deterministic stability result, no index changes its assigned center. Since the assignments remain unchanged, all pairwise relations are preserved, and therefore
$P' = P.$
This example illustrates the role of the minimum margin as a sufficient condition for stability.

\subsection{Near-boundary configuration}

We now consider configurations in which some indices have small margin, meaning that the corresponding points lie close to the decision boundary. Let $\varepsilon > 0$. By the structural bound, any index $i$ that changes its assigned center must satisfy
$\gamma_i \le 2\varepsilon.$
Thus assignment changes are localized near the boundary. Points with large margin remain stable, while points with small margin are sensitive to perturbations. Depending on how these points are arranged, assignment changes may lead to changes in the induced partition.

\subsection{Many-point switching configuration}

We describe a configuration in which many indices change their assigned centers under a small perturbation. Let $m \ge 1$, and define labeled points indexed by $[m+2]$ as follows:
\[
x_1 = (-2,0), \qquad x_2 = (2,0), \qquad x_{2+i} = (\delta,i), \quad i=1,\ldots,m,
\]
where $\delta > 0$ is small. The induced partition is
\[
P = \{\{1\}, \{2,3,\ldots,m+2\}\}.
\]
Define a perturbation by
\[
x_{2+i}' = (-\delta,i), \quad i=1,\ldots,m,
\]
and keep $x_1,x_2$ fixed. Then $\|X - X'\| = 2\delta$, which can be made arbitrarily small, while the induced partition becomes
\[
P' = \{\{1,3,\ldots,m+2\}, \{2\}\}.
\]
In this case, many indices change their assigned centers, and this produces a large change in the partition.

\subsection{Stability threshold behavior}

Define
\[
S(\varepsilon) = d\big(A(X), A(X')\big).
\]
From the deterministic bound, we expect the following qualitative behavior. If $\varepsilon < \gamma_{\min}/2$, then $S(\varepsilon) = 0$. Once $\varepsilon$ exceeds this threshold, assignment changes may occur, and $S(\varepsilon)$ may increase. The increase is driven by indices with small margin.

This illustrates how the threshold $\gamma_{\min}/2$ governs when changes in the partition can occur.

\subsection{Gaussian perturbations}

Under Gaussian perturbations, the probability that index $i$ changes its assigned center is bounded by
\[
\mathbb{P}(\|\eta_i\| \ge \gamma_i/2).
\]
Thus indices with large margin are unlikely to switch, while indices with small margin are more sensitive to noise. As the noise level increases, assignment changes become more frequent, and this may lead to gradual changes in the partition.

\subsection{Discrete-time evolution}

Consider a sequence of configurations
\[
X(0), X(1), \ldots, X(T),
\]
with one-step perturbations $\delta_t = \|X(t+1) - X(t)\|$. If
\[
\sum_{r=0}^{t-1} \delta_r < r(X(0),A),
\]
then by the discrete-time stability result,
\[
A(X(t)) = A(X(0)).
\]
This shows that stability over time is controlled by cumulative perturbation rather than individual steps.

These examples illustrate how the margin governs stability. Indices corresponding to points far from decision boundaries remain stable, while indices near boundaries are the primary source of assignment changes. Depending on the configuration, these assignment changes may or may not lead to changes in the induced partition.

\section{Discussion and Limitations}

In the framework of this paper, the induced partition is determined by the grouping of indices according to their assigned centers. More precisely, two indices lie in the same block if and only if their corresponding points are assigned to the same center under the fixed nearest-center rule.

If no index changes its assigned center under a perturbation, then the induced partition remains unchanged. However, the converse does not hold in general. It is possible for assignment changes to occur without changing the partition, since partitions record only which indices are grouped together and do not encode the labels of the centers. For example, simultaneous reassignment of entire groups from one center to another may preserve the partition structure. For this reason, the analysis distinguishes between assignment changes and changes in the induced partition.

The framework treats clustering as a partition of a finite labeled set. This allows the problem of stability to be expressed in terms of discrete objects such as partitions and pairwise relations. The definitions of partition distance and stability radius provide direct ways to measure how the output changes under perturbation.

The main results are developed for nearest-center assignment with fixed centers in Euclidean space. In this setting, the margin of each point provides a sufficient condition for stability and yields a lower bound on the stability radius. If all points have sufficiently large margin, then small perturbations do not change any assigned centers, and hence the induced partition remains unchanged. When many points have small margin, even small perturbations may lead to assignment changes, and depending on how these changes are distributed, the induced partition may also change.

The results identify where instability arises. Points that lie close to decision boundaries are the primary source of assignment changes. The constructions in Section 6 show that when such points are present, arbitrarily small perturbations can change one or many assignments. In those constructions, these assignment changes lead to changes in the induced partition, but this behavior depends on the structure of the configuration.

The discrete-time extension shows how stability can be maintained over multiple steps. If the cumulative perturbation remains below the initial stability radius, then the initial partition is preserved. The stepwise condition gives a complementary local view, where stability is maintained at each step provided the perturbation is small relative to the current configuration.

The probabilistic extension relates the margin to average behavior under random perturbations. Points with large margin are unlikely to change their assigned centers, while points with small margin have higher probability of switching. The expected change in the partition can be bounded in terms of these switching probabilities, although assignment changes do not always translate directly into proportional changes in the partition.

There are several limitations to this framework. The analysis is restricted to nearest-center assignment with fixed centers. It does not include randomness in the clustering rule itself, such as random initialization in $k$-means, and does not address optimization effects. The framework also focuses on Euclidean distance and does not treat more general metrics or graph-based clustering methods. In addition, high-dimensional effects are not studied.

The results should therefore be viewed as a simplified model that isolates the role of margin in stability. Within this setting, the framework gives explicit sufficient conditions under which clustering-based decisions are stable and identifies which points are most sensitive to perturbations.

\section{Conclusion}

We studied clustering through the partitions it induces on a finite labeled set and developed a framework to describe how these partitions change under perturbations of the data.

We introduced a partition distance and a notion of stability radius, and analyzed these quantities in the setting of nearest-center assignment with fixed centers. In this setting, we showed that the margin of each point provides a sufficient condition for stability and yields a lower bound on the stability radius. When all points have sufficiently large margin, no point changes its assigned center under small perturbations, and the induced partition remains unchanged.

We also constructed explicit examples demonstrating that instability can occur when points lie close to decision boundaries. In these examples, assignment changes lead to changes in the induced partition. More generally, assignment changes may or may not alter the partition, depending on how the grouping structure is affected.

We extended the framework to a discrete-time setting and gave conditions under which the partition remains unchanged over multiple steps, based on cumulative perturbations. We also gave a probabilistic extension that relates the likelihood of assignment changes to the margin and provides bounds on expected changes in the partition.

The main perspective of this work is that clustering can be studied through the partitions it produces, and that stability can be expressed in terms of how these partitions change under perturbation. In the setting considered here, the margin provides a direct way to control both worst-case and average behavior.

The results are limited to a fixed-center Euclidean model, but they give a clear starting point for further study of stability in clustering-based methods.

\appendix
\section{Appendix}

This appendix collects supporting arguments and technical details used in the main text. We include a detailed version of the boundary-margin argument, a bound on partition distance, and a brief note on the computational setup.

\subsection{Boundary-margin criterion (detailed form)}

We restate the boundary-margin criterion in expanded form.

\begin{lemma}
Let
\[
X = (x_1,\ldots,x_n), \qquad X' = (x_1',\ldots,x_n') \in (\mathbb{R}^d)^n
\]
satisfy
\[
\|X - X'\| \le \varepsilon.
\]
If
\[
\varepsilon < \frac{\gamma_i}{2},
\]
then the point $x_i$ does not change its assigned center.
\end{lemma}

\begin{proof}
Let $\ell_i$ denote the index of the center to which $x_i$ is assigned. For any center $c_j$, by the triangle inequality,
\[
\|x_i' - c_j\| \le \|x_i - c_j\| + \varepsilon, \qquad
\|x_i' - c_j\| \ge \|x_i - c_j\| - \varepsilon.
\]

For any $j \ne \ell_i$, we compare distances:
\[
\|x_i' - c_j\| - \|x_i' - c_{\ell_i}\|
\ge (\|x_i - c_j\| - \varepsilon) - (\|x_i - c_{\ell_i}\| + \varepsilon).
\]
Thus
\[
\|x_i' - c_j\| - \|x_i' - c_{\ell_i}\|
\ge (\|x_i - c_j\| - \|x_i - c_{\ell_i}\|) - 2\varepsilon
\ge \gamma_i - 2\varepsilon.
\]

If $\varepsilon < \gamma_i/2$, then $\gamma_i - 2\varepsilon > 0$. Hence for all $j \ne \ell_i$,
$\|x_i' - c_j\| > \|x_i' - c_{\ell_i}\|,$
so $x_i'$ remains assigned to $c_{\ell_i}$.
\end{proof}

\subsection{A bound on partition distance}

We give an upper bound on the partition distance in terms of the number of indices that change their assigned centers.

\begin{lemma}
Let $P = A(X)$ and $Q = A(X')$ be partitions produced by the same nearest-center assignment rule with fixed centers and a deterministic tie-breaking rule. If at most $m$ indices change their assigned centers between $X$ and $X'$, then
$d(P,Q) \le \min\!\left(1, \frac{2m}{n-1}\right).$

\end{lemma}

\begin{proof}
Each index that changes its assigned center can affect at most $(n-1)$ unordered pairs involving that index. Thus the total number of pairs whose same-block relation may change is at most $m(n-1)$.

Dividing by $\binom{n}{2}$ gives
\[
d(P,Q) \le \frac{m(n-1)}{\binom{n}{2}} = \frac{2m}{n-1}.
\]
Since $d(P,Q) \le 1$ by definition, the stated bound follows.
\end{proof}

This bound controls how changes in assignments can affect the partition distance. In general, assignment changes do not always produce maximal changes in partition distance, since multiple assignment changes may preserve or partially preserve the grouping structure.

\subsection{Computational details}

All numerical illustrations use synthetic datasets in $\mathbb{R}^2$. Points are generated near fixed centers as described in the main text. Perturbations are applied either as bounded noise, where $\|\eta_i\| \le \varepsilon$, or as Gaussian noise, where $\eta_i \sim \mathcal{N}(0,\sigma^2 I_2)$ independently. Clustering is performed using nearest-center assignment with fixed centers and a deterministic tie-breaking rule. The partition distance is computed using the pairwise disagreement measure defined in Section 3. All quantities, such as $N_{\mathrm{sw}}$ and $d(P,P')$, are computed directly from the definitions. In simulations, these quantities are averaged over independent trials to illustrate typical behavior.

\section*{Acknowledgment}
The authors are grateful to Professor Michael R. Yaraturo (Department of Mathematics, Penn State University Brandywine) for a careful reading of the manuscript and for valuable suggestions that improved the exposition.

\end{document}